\documentclass[10pt,a4paper]{article}
\usepackage{graphicx,epstopdf} 
\usepackage[caption=false]{subfig} 
\usepackage{hyperref}
\usepackage{graphicx,color}
\usepackage{amsmath,amssymb}
\usepackage{amsthm}
\usepackage{fullpage}
\usepackage{mathabx}
\usepackage{biblatex}
 \newtheorem{theorem}{Theorem}[section]
 \newtheorem{lemma}[theorem]{Lemma}
 \newtheorem{proposition}[theorem]{Proposition}

 \newtheorem{definition}{Definition}

 \newtheorem{assumption}{Assumption}

\newcommand{\real}{\mathbb{R}} 
\newcommand{\zDomain}{\mathcal{Z}}
\newcommand{\Domain}{\mathcal{X}\times\mathcal{Y}}
\newcommand{\realDomain}{\mathbb{R}^{n}\times\mathbb{R}^{m}}
\newcommand{\argmax}[2] {\mathrm{arg}\max_{#1}#2}
\newcommand{\argmin}[2] {\mathrm{arg}\min_{#1}#2}
\newcommand{\norm}[1]{\Vert #1 \Vert}

\DeclareMathOperator{\EX}{\mathbb{E}}
\DeclareMathOperator{\sW}{subW}

\DeclareMathOperator{\dist}{d}

\title{Online Saddle Point Tracking with Decision-Dependent Data}
\date{}
\author{Killian Wood\thanks{Department of Applied Mathematics, University of Colorado, Boulder, CO (killian.wood@colorado.edu).} 
\ and Emiliano Dall'Anese\thanks{Department of Electrical, Computer, and Energy Engineering, and Department of Applied Mathematics, University of Colorado, Boulder, CO (emiliano.dallanese@colorado.edu).}
}

\begin{document}
\maketitle
\begin{abstract}%
In this work, we consider a time-varying stochastic saddle point problem in which the objective is revealed 
sequentially, 
and the data distribution depends on the decision variables. Problems of this type express the distributional dependence via a distributional map, and are known to have two distinct types of solutions\textemdash saddle points and equilibrium points. We demonstrate that, under suitable conditions, online primal-dual type algorithms are capable of tracking equilibrium points. In contrast, since computing closed-form gradient of the objective requires knowledge of the distributional map, we offer an online stochastic primal-dual algorithm for tracking equilibrium trajectories. We provide bounds in expectation and in high probability, with the latter leveraging a sub-Weibull model for the gradient error. We illustrate our results on an electric vehicle charging problem where responsiveness to prices follows a location-scale family based distributional map. 
\end{abstract}


\section{Introduction}
The general goal of stochastic optimization is to find optimal decisions in systems with parameters dictated by data~\cite{nemirovski2009robust,shapiro2005complexity,zhang2021generalization}. In statistical learning, optimal decisions represent model parameters that best fit a mapping between feature and label data (see, e.g.,~\cite{gurbuzbalaban2020heavy,simsekli2019tail}). In the context of optimization of physical and dynamical systems, they may model externalities or system parameters that are predicted from data and are accompanied by given error statistics (see, e.g.,~\cite{berberich2020data,bianchin2021online,li2021online,}). To analyze these problems, works posit that the data distributions are stationary~\cite{birge2011introduction}; in modern machine learning and cyber-physical systems applications this assumption may be violated when population data shifts in response to previously deployed decisions, thus making said decisions sub-optimal. Hence the distribution is inextricably tied to the decision variables. 

This work considers the problem of tracking the solution trajectories for problems of the form:
\begin{equation}
\label{eq:problem_statement}
    \min_{x\in\mathcal{X}_{t}}\max_{y\in\mathcal{Y}_{t}} \left\{ F_{t}(x,y): = \underset{w\sim D_{t}(x,y)}{\EX}[f_{t}(x,y,w)] \right\}
\end{equation}
where $t$ is a time index, $\mathcal{X}_{t}\subseteq \real^{n}$ and $\mathcal{Y}_{t}\subseteq \real^{m}$ are convex and compact sets capturing time-varying constraints, $f_{t}:\realDomain\times\real^{k}\rightarrow \real$ is a strongly-convex-strongly-concave function revealed at time $t$, and $D_{t}: \real^{d}\times\real^{n}\rightarrow \mathcal{P}(\real^{k})$ is a distributional map that maps decision variables to the set of finite-first moment probability distributions supported on $\real^{k}$ denoted by $\mathcal{P}(\real^{k})$. Without loss of generality, we refer to the support of $w$ as $\real^{k}$ (even if $w$ is matrix valued, our analysis holds as $w$ is isomorphic to its vectorization over $\real^{k}$).


Examples of problems of the form~\eqref{eq:problem_statement} emerge in cost maximization in competitive markets, where the (stochastic) demand shifts in response to prices (see, e.g.,~\cite{maheshwari2021zeroth,turan2021competition}), and in applications in adversarial strategic classification, finance, energy systems, transportation networks, and ride-sharing\textemdash just to mention a few. Focusing on the first example, consider  a competition between two service providers in an area with $n$ distinct regions for which each provider seeks to maximize their relative revenue, and when the demand for each provider's service changes in response to the price variation set by both providers. This problem can be written as the saddle point problem 
\begin{equation}
\label{application:charging_market} 
    \min_{x\in\mathcal{X}_{t}} \max_{y\in\mathcal{Y}_{t}} \left\{ F_{t}(x,y) = \EX_{(a,b)\sim D_{t}(x,y) }
    \norm{\Gamma_{t}^{1} x}^{2} - \norm{\Gamma_{t}^{2} y}^{2} - \langle a + c_{t}, x \rangle + \langle b + c_{t} , y \rangle \right\} \, ,
\end{equation}
where $x=(x_{i})_{i=1}^n$ and $y=(y_{i})_{i=1}^n$ are vectors of price deviations from a nominal value for providers one and two respectively (components $x_{i}$ and $y_{i}$ are the prices in region $i\in[n]$); $\Gamma_{t}^{1},\Gamma_{t}^{2}\in\real^{n\times n}$ are the charging rate utility matrices; $c_{t}\in\real^{n}$ is the location-based utility vector (i.e., cost of operation); and $a,b\in\real^{n}$ are changes in demand in each region (in response to price changes) with distributions $a \overset{d}{=} a^{t}_{0} + A^{t}_{1}x + B^{t}_{1}y$, and $b \overset{d}{=} b^{t}_{0} + A^{t}_{2}x + B^{t}_{2}y$. Here $a^{t}_{0}$ and $b^{t}_{0}$ are random variables drawn from zero-mean stationary distributions.

Classical solutions to~\eqref{eq:problem_statement} are saddle points, which we denote $z_{t}^{*}=(x_{t}^{*}, y_{t}^{*})\in \mathcal{X}_{t} \times \mathcal{Y}_{t}$. Under appropriate conditions, namely minimax equality, saddle points satisfy 
\begin{equation}
    x_{t}^{*}  \in \argmin{x\in\mathcal{X}_{t}}\max_{y\in\mathcal{Y}_{t}} F_{t}(x,y), 
    \ \ 
    y_{t}^{*}  \in \argmax{y\in\mathcal{Y}_{t}} \min_{x\in\mathcal{X}_{t}} F_{t} (x,y).
\end{equation}
In this setting, saddle points are optimal decisions that effectively anticipate the distributional shift, and hence are optimal even after the data distribution has changed in the system. While these are ideal, finding them is typically computationally intractable. While sufficient conditions for their existence and uniqueness have been studied, guarantees for convergence to saddle points are only approximate or require explicit knowledge of a model for the distributional map \cite{narang2022learning, wood2022saddle}. A common heuristic to overcome distributional shift in general is to repeatedly retrain the optimal decisions each time the distribution shifts. This amounts to forming a sequence $\{z_{t}^{\ell}\}_{\ell\geq 0}=\{(x_{t}^{\ell}, y_{t}^{\ell})\}_{\ell\geq 0}$ at each time $t$ defined by 
\begin{equation}
    \begin{aligned}
    \label{eqn:repeated_retraining}
   & x_{t}^{\ell+1} \in \argmin{x\in\mathcal{X}_{t}}\max_{y\in\mathcal{Y}_{t}} \underset{w\sim D_{t}(x_{t}^{\ell},y_{t}^{\ell})}{\EX} [f_{t}(x,y,w)], \\ 
   & y_{t}^{\ell+1} \in \argmax{y\in\mathcal{Y}_{t}}\min_{x\in\mathcal{X}_{t}} \underset{w\sim D_{t}(x_{t}^{\ell},y_{t}^{\ell})}{\EX} [ f_{t}(x,y,w)].
    \end{aligned}
\end{equation}

The fixed points of this repeated retraining procedure have been coined \textit{equilibrium points}, and are known to exist under mild conditions. In what follows we provide algorithms capable of tracking the equilibrium point trajectory $\{\bar{z}_{t}\} = \{\bar{x}_{t}, \bar{y}_{t}\}$ without requiring that we take the sequences in \ref{eqn:repeated_retraining} to convergence ($\ell\to\infty)$. This will be crucial for our online setting, as we assume that each time $t$, a new function and distributional map arrive (\cite{besbes2015non,cao2020online,dall2020optimization,jadbabaie2015online,shames2020online,wood2021online}). 

\subsection{Related Work}

\textbf{Stochastic Saddle Point Problems.} Algorithms for computing saddle points can be loosely catagorized as primal-dual based  or proximal based~\cite{ koshal2011multiuser, mokhtari2020unified, nemirovski2004prox, nemirovski2009robust, zhang2021robust}. 
Some works seek to find approximate saddle points by analyzing a saddle point gap. We conduct our analysis in a setting in which solutions are known to be unique, so we simply track them. Our analysis is primarily conducted through variational analysis \cite{rockafellar2009variational}. Hence we define the appropriate gradient maps and demonstrate that solutions to our problems are the solution to the variational inequalities induced by said gradient maps. 

\textbf{Decision Dependent Distributions.} This work is most closely related to the literature on stochastic optimization with decision dependent distributions, or its counterpart in learning, performative prediction. The problem of finding optimal decisions that are robust to decision-dependent data has been studied extensively, and in many distinct settings: minimization problems \cite{drusvyatskiy2022stochastic,perdomo2020performative}, saddle-point problems~\cite{wood2022saddle}, games \cite{narang2022learning}, online \cite{wood2021online}, time-varying decay \cite{ray2022decision}. Relative to the existing work on saddle point problems in the literature, this work considers problems for which the objective, constraints and distribution are time-varying and revealed sequentially in time. The work on games is related, as specific instances of games such as two-player zero-sum may be cast into a saddle point problem. Saddle point problems however are however not a strict subset of games as they exist in their own right; arising from constrained minimization problems, etc. The most obvious inspiration for this work is that of \cite{wood2022saddle}, as the setting in this work is precisely stochastic saddle point problems with decision-dependent distributions for time-invariant problems. Relative to this work, the results we present here are extensions to the online setting where analysis requires handling of additional noise due to solution drift. 

\textbf{Online Convex Optimization.} Relevant to works on online optimization that are concerned with tracking trajectories (see the representative works \cite{,cutler2021stochastic,mokhtari2016online,madden2021bounds,popkov2005gradient,selvaratnam2018Numerical}) or given comparator sequences~\cite{jadbabaie2015online}; another line of work is concerned with finding a sequence of decision that minimize a suitable dynamic regret metric \cite{hazan2019introduction}. Our metric is the distance to the solution of~\eqref{eq:problem_statement} at the current iteration, where we incorporate the drift of the solution trajectory. We account for the time-variability of the solution by incorporating the solution drift into our guarantees. 

\subsection{Contributions} Our contributions are as follows.

\noindent \emph{(c1)} \textit{The Online Equilibrium Problem}. We propose a notion of equilibrium points for the time-varying saddle-point problem in \ref{eq:problem_statement}, provide conditions to guarantee existence and uniqueness, and provide bounds for the distance between the unique equilibrium points~\eqref{eq:problem_statement}.

\noindent \emph{(c2)} \textit{Online Algorithms}. We demonstrate that primal-dual algorithms, using the gradients of $f_{t}$, are effective at finding equilibrium points when the stochastic objective $f_{t}$ is strongly-convex-strongly-concave for any realization of $w$. First, we demonstrate effective tracking of a conceptual algorithm using full gradient information. We then demonstrate that a stochastic algorithm tracks equilibrium points with additional noise due to estimation. Furthermore, we provide provide expectation bounds and high probability bounds that hold for each iteration.


\noindent \emph{(c3)} \textit{Experiments}. We illustrate our results on the electric vehicle charging problem in \eqref{application:charging_market} by incorporating synthetic demand data from \cite{gilleran2021electric}. Here, the demand changes in response to prices with a location-scale family based distributional map.


\section{Equilibrium Points}

In this section we define the equilibrium problem, the fixed points of the repeated retraining heuristic in \ref{eqn:repeated_retraining}, and provide sufficient conditions for their existence. We start from the definition of equilibrium points. 
\begin{definition}{\textbf{\textit{(Equilibrium Points)}}} 
\label{def:equilibrium}
A pair $(\bar{x}_{t},\bar{y}_{t})\in\Domain$ is an equilibrium point if:
\begin{equation}
    \begin{aligned}
    & \bar{x}_{t} \in \arg\min_{x\in\mathcal{X}_{t}} \left\{\max_{y\in\mathcal{Y}_{t}} \underset{w\sim D_{t}(\bar{x}_{t},\bar{y}_{t})}{\EX} [f_{t}(x,y,w)] \right\}, \\
    & \bar{y}_{t} \in \arg\max_{y\in\mathcal{Y}_{t}}\left\{\min_{x\in\mathcal{X}_{t}} \underset{w\sim D_{t}(\bar{x}_{t},\bar{y}_{t})}{\EX} [ f_{t}(x,y,w)] \right\}.
    \end{aligned}
\end{equation}
Sequences of equilibrium points are defined as $(\bar{x}_{t}, \bar{y}_{t})_{t \in \mathbb{N}}$. \hfill $\Box$
\end{definition}
In essence, equilibrium points are the solutions to the stationary saddle point problem that they induce. In this way, they are optimal decisions when data distribution is in state $D_{t}(\bar{x}_{t},\bar{y}_{t})$ but need not be optimal otherwise. Existence of these points is contingent on the distributional function being continuous on the set of probability distributions, and $f_{t}$ being at least convex-concave.

\begin{theorem}{(\textit{\textbf{Existence of Equilibrium Points}})}
\label{thm:existence}
Suppose that the following assumptions hold at time $t\geq0$:  

\noindent i) $x \mapsto f_{t}(x,y,w)$ is convex in $x$ for all $y\in\mathcal{Y}_{t}$ and for any realization of $w$; 

\noindent ii) $y \mapsto f_{t}(x,y,w)$ is concave in $y$ for all $x\in\mathcal{X}_{t}$ and for any realization of $w$; 

\noindent iii) $(x,y) \mapsto f_{t}(x,y,w)$ is continuous on $\Domain$ for any given $w$; 

\noindent iv) the sets $\mathcal{X}_{t}\subset\real^{n},\mathcal{Y}_{t}\subset\real^{n}$ are convex compact subsets; 

\noindent v) the distributional map $D_{t}:\zDomain\to (\mathcal{P}(M),W_{1})$ is continuous. 

\noindent Then the set of equilibrium points is nonempty and compact.  \hfill $\Box$
\end{theorem}
The proof follows from the fact that 
equilibrium points exist for each problem at time $t$ due to \cite[Theorem 2.10]{wood2022saddle}. The proof strategy amounts to demonstrating that the repeated retraining map satisfies Kakutani's Fixed Point Theorem; see, e.g., \cite[Corollary 17.55]{guide2006infinite}; it is not provided due to space limitations.

\subsection{Theoretical Framework}

In light of our discussion on existence of equilibrium points, we outline the assumptions and some results that will be necessary later in our analysis. For notational convenience, we will refer to the stacked variable $z = (x,y)$ and the Cartesian product set of constraints $\zDomain=\Domain$. We will rely on the following assumptions to hold at each time $t$ throughout this work. 

\begin{assumption}{(\textit{\textbf{Strong-Convexity-Strong-Concavity}})}\label{assumption_one:strong}
The function $(x,y) \mapsto f_{t}(x,y,w)$ is continuously differentiable over $\realDomain$ for any realization of $w$. The function $(x,y) \mapsto f_{t}(x,y,w)$ is  $\gamma$-strongly-convex-strongly-concave, for any realization of $w$; that is, $f_{t}$ is $\gamma$-strongly-convex in $x$ for all $y\in\real^{m}$ and $\gamma$-strongly-concave in $y$ for all $x\in\real^{n}$. 
\hfill $\Box$
\end{assumption}

\begin{assumption}{(\textit{\textbf{Joint Smoothness}})}\label{assumption_two:smooth}
The map $g_{t}(z,w) := (\nabla_{x}f_{t}(z,w),-\nabla_{y} f_{t}(z,w))$ is $L$-Lipschitz in $z$ and $w$. Namely, 
    $\norm{ g_{t}(z,w) - g_{t}(z',w)  }  \leq L \norm{z-z'}$, 
    $\norm{ g_{t}(z,w) - g_{t}(z,w')  }  \leq L \dist(w,w')$, 
for any $z,z'\in\realDomain$ and $w,w'$ supported on $\real^{k}$, for some $L \geq 0$, where $\dist:\real^{k}\times \real^{k} \rightarrow \real$ is some chosen metric on $\real^{k}$. 
\hfill $\Box$
\end{assumption}

\begin{assumption}{(\textit{\textbf{Lipschitz-Continuous Distributional Map}})} \label{assumption_three:distribution_map}
The distributional maps \\ $D_{t}:\realDomain \rightarrow\mathcal{P}(M)$ are $\varepsilon$-Lipschitz. Namely, 
    $W_{1}(D_{t}(z),D_{t}(z'))\leq \varepsilon\norm{z-z'}$, 
for any $z,z'\in\realDomain$, where $W_{1}$ is the Wasserstein-1 distance.  
\hfill $\Box$
\end{assumption}

\begin{assumption}{(\textit{\textbf{Compact Sets}})} 
\label{assumption_four:sets}
The sets $\mathcal{X}_t\subset\real^{n}$ and $\mathcal{Y}_{t}\subset\real^{m}$ are compact and convex. 
\hfill $\Box$
\end{assumption}

\begin{assumption}{(\textit{\textbf{Bounded Drift}})} 
\label{assumption_five:bounded_drift}
There exists a $\Delta>0$ such that the equilibrium drift sequence defined by $\Delta_{t} : = \norm{\bar{z}_{t+1} - \bar{z}_{t}}$ is uniformly bounded by $\Delta$. Namely, $\Delta_{t}\leq \Delta$ for all $t\geq 0$.
\hfill $\Box$
\end{assumption}

These assumptions provided are sufficient to guarantee uniqueness of the equilibrium point, and convergence of primal-dual algorithms in the batch setting; see \cite{wood2022saddle}.

\begin{theorem}{(\textit{\textbf{Equilibrium Point Uniqueness}})} If Assumptions \ref{assumption_one:strong}-\ref{assumption_four:sets} are satisfied such that $\varepsilon L < \gamma$, then a unique equilibrium point exists.
\end{theorem}

Proof of this results amounts to showing that the repeated retraining heuristic in \ref{eqn:repeated_retraining} is a strict contraction and hence satisfies the Banach-Picard Fixed Point Theorem. For a detailed proof, see \cite{wood2022saddle}.

Given that the data distribution is shifting, it is necessary to characterize this shift and its effect on the gradient. The key to computing equilibrium points will be the gradients of $f_{t}$. We note that this is only one term required to compute the gradients of $F_{t}$, effectively ignoring the dependence of $D_{t}$ on the decision variables. For now, we will denote the decoupled gradient map as the function $G_{t}$ defined by 
\begin{equation}
\label{eqn:equilirbium_gradient}
    G_{t}(z;z') := \underset{w\sim D_{t}(z')}{\EX} g_{t}(z,w) = \left(
    \underset{w\sim D_{t}(z')}{\EX} \nabla_{x} f_{t}(z,w),  \underset{w\sim D_{t}(z')}{\EX} -\nabla_{y}f_{t}(z,w)
    \right)    
\end{equation}
for all $z,z'\in\realDomain$. Note that we refer to this gradient map as ``decoupled'' as we separate the decision variable in the stochastic objective and the distributional map. This will allow us to characterize these behaviors separately.
\begin{lemma}{(\textit{\textbf{Gradient Map Characterization}})}
\label{lemma:gradient_characterization}
If Assumptions \ref{assumption_one:strong}-\ref{assumption_four:sets} hold, then:
\begin{enumerate}
    \item (\textbf{Gradient Deviation}) For any fixed $\hat{z}\in\zDomain$, the map $z \mapsto G_{t}(\hat{z},z)$ is $\varepsilon L$-Lipschitz-continuous. That is, $\norm{G_{t}(\hat{z};z)-G_{t}(\hat{z};z')}\leq \varepsilon L \norm{z-z'}$, 
    for all $z,z'\in\zDomain$.
    \item (\textbf{Strong-Monotonicity}) The map $z\mapsto G_{t}(z,z)$ is $(\gamma-\varepsilon L)$-strongly-monotonic.

    \item (\textbf{Lipschitz-Continuity}) The map $z\mapsto G_{t}(z,z)$ is $(L+\varepsilon L)$-Lipschitz Continuous. \hfill $\Box$
\end{enumerate}
\label{lem:gradient_deviations}
\end{lemma}
Proof of the Gradient Deviation property follows by combining the properties allowed from joint smoothness and lipschitz continuity of the distributional map (Assumptions \ref{assumption_two:smooth} and \ref{assumption_three:distribution_map} respectively. For a detailed proof, we refer the reader to \cite{wood2022saddle}. Strong monotonicity and Lipschitz continuity of $z\mapsto G_{t}(z;z)$  then follow immediately. With this lemma, we can effectively deal with the decoupled gradient map by passing variables into both the $D_{t}$ and $g_{t}$ simultaneously. Going forward, we will simply write $G_{t}$ to mean the gradient map given by $z\mapsto G_{t}(z;z)$.

\section{Online Algorithms}

\subsection{A Conceptual Primal-Dual Algorithm}
In this section, we show that if the decoupled gradient map $G_{t}$ is available, then tracking the equilibrium points is possible using a primal-dual algorithm. This provides a basis of comparison for our analysis in the next section where we use a stochastic gradient estimator in place of $G_{t}$. Indeed, we denote the equilibrium primal-dual algorithmic map by
\begin{equation}
    \label{eqn:primal_dual_map}
    \mathcal{G}_{t}(z) = \Pi_{\zDomain}\left( 
    z - \eta G_{t}(z)
    \right)
\end{equation}
so that the algorithm generates the sequence $\{z_{t}\}_{t\geq 0}$ defined by $z_{t+1} = \mathcal{G}_{t}(z_{t})$, $t \in \mathbb{N}$. 
To proceed, we observe that equilibrium points are the fixed points of the primal-dual algorithmic map.

\begin{proposition}{\textit{\textbf{(Fixed Point Characterization)}}}
\label{prop:equilibrium_fixed_point}
Let Assumptions \ref{assumption_one:strong}-\ref{assumption_four:sets} hold and suppose that $\frac{\varepsilon L }{\gamma}<1$. A point $\bar{z}_{t}\in\zDomain$ is an equilibrium point if and only if $\bar{z}_{t}=\mathcal{G}_{t}(\bar{z}_{t})$.
\hfill $\Box$
\end{proposition}

This proposition will allow us to cast our analysis into a fixed point framework, using the equilibrium points as the fixed points of the distributional map. 

\begin{theorem}{(\textbf{\textit{Primal-Dual Tracking}})}
 Suppose that Assumptions \ref{assumption_one:strong}-\ref{assumption_four:sets} hold and that $\frac{\varepsilon L }{\gamma}<1$. Then the sequence $z_{t+1} = \mathcal{G}_{t}(z_{t})$ satisfies the bound 
\begin{equation}
\label{thm:equilibrium_contraction}
    \norm{z_{t} - \bar{z}_{t}} \leq \alpha^{t} \norm{z_{0}-\bar{z}_{0}} + (1-\alpha)^{-1} \Delta
\end{equation}
for any initial point $z_{0}\in\zDomain$, and  $\alpha :=\sqrt{1-\eta(\gamma-\varepsilon L)}$ provided that 
\begin{equation}
    \label{thm:convergence_stepsize_condition}
    \eta < \min \left\{ \frac{1}{\gamma - \varepsilon L} , 
    \frac{\gamma - \varepsilon L}{(1+\varepsilon)^{2}L^{2}} \right\}
\end{equation}
Furthermore, $\{z_{t}\}_{t \geq 0}$ ultimately tracks the sequence of unique equilibrium points $\{\bar{z}_{t}\}_{ t \geq 0}$ in the sense that
$\limsup_{t\to\infty}\norm{z_{t}-\bar{z}_{t}}\leq (1-\alpha)^{-1}\Delta$. 
\hfill $\Box$
\end{theorem}

\begin{proof} It follows from the triangle inequality that 
$\norm{z_{t+1}-\bar{z}_{t+1} }
\leq \norm{z_{t+1} - \bar{z}_{t}} + \norm{\bar{z}_{t}-\bar{z}_{t+1}} 
= \norm{z_{t+1} - \bar{z}_{t}} + \Delta_{t}$,
and hence we simply need to bound $\norm{z_{t+1}-\bar{z}_{t}}$. We observe that 
    \begin{align*}
        \norm{z_{t+1}-\bar{z}_{t}}^{2} 
        & =  \norm{\Pi_{\zDomain}\left(z_{t}-\eta G_{t}(z_{t})\right) - \Pi_{\zDomain}\left(\bar{z}_{t}-\eta G_{t}(\bar{z}_{t})\right) }^{2} \\ 
        & \leq \norm{(z_{t}-\bar{z}_{t}) - \eta\left(G_{t}(z_{t})-G_{t}(\bar{z}_{t})\right)}^{2} \\ 
        & \leq \norm{z_{t}-\bar{z}_{t}}^{2} - 2\eta\langle z_{t} - \bar{z}_{t}, G_{t}(z_{t}) - G_{t}(\bar{z}_{t}) \rangle + \eta^{2}\norm{G_{t}(z_{t})-G_{t}(\bar{z}_{t})}^{2}.
    \end{align*}
If we denote $\hat{\gamma} = \gamma - \varepsilon L$ and $\hat{L} = L + \varepsilon L $, then from Lemma \ref{lemma:gradient_characterization} we have that $G_{t}$ is $\hat{\gamma}$-strongly monotone and $\hat{L}$-Lipschitz continuous. Combining these facts yields
\begin{equation*}
    \langle z_{t} - \bar{z}_{t}, G_{t}(z_{t}) - G_{t}(\bar{z}_{t}) \rangle 
    \geq \frac{\hat{\gamma}}{2}\norm{z_{t}-\bar{z}_{t}}^{2} + \frac{\hat{\gamma}}{2\hat{L}^{2}} \norm{ G_{t}(z_{t})-G_{t}(\bar{z}_{t})}^{2}.
 \end{equation*}
Substituting into the above yields
 \begin{align*}
 \norm{z_{t+1}-\bar{z}_{t}}^{2} 
 & \leq (1-\eta \hat{\gamma})\norm{z_{t}-\bar{z}_{t}}^{2} + 
 \eta\left( \eta - \frac{\hat{\gamma}}{\hat{L}^{2}}\right)\norm{G_{t}(z_{t})-G_{t}(\bar{z}_{t})}^{2} \leq  (1-\eta \hat{\gamma})\norm{z_{t}-\bar{z}_{t}}^{2} 
 \end{align*}
where the last inequality follows provided that $\eta\leq \hat{\gamma} / \hat{L}^{2} $. It follows that if $\eta < 1 / \hat{\gamma}$ as well, then $ 1 - \eta\hat{\gamma}< 1$ and the bound in Theorem \eqref{thm:equilibrium_contraction} follows. Considering the limit supremum of the bound in \eqref{thm:equilibrium_contraction} yields the result.
\end{proof}
We note that the noise due to the drift in \eqref{thm:equilibrium_contraction} increases as  we decrease the step size $\eta$. Hence it is impossible to completely remove this disturbance from the algorithm. This reflects intuition however as very small step sizes would make it difficult to ever reach the solution trajectory. Meanwhile, larger step sizes decrease this noise while simultaneously decreasing the rate at which we overcome the error $z_{t+1}-\bar{z}_{t}$ between successive iterates. We build on this intuition in our stochastic algorithm. This concludes our discussion of the conceptual primal-dual algorithm. In the next section, we demonstrate tracking of a stochastic primal-dual algorithm.

\subsection{A Stochastic Primal-Dual Algorithm}
In previous section, we demonstrated that a conceptual first-order algorithm is capable of tracking the trajectory of equilibrium point. We say conceptual because having access to full information in $G_{t}$ requires the ability to compute the expectation with respect to the distributional map $D_{t}$ at each algorithmic step\textemdash which is of course impractical. Hence, we are concerned with a more pragmatic setting in which we merely have access to a stochastic gradient oracle, which we will denote $H_{t}$. We make the implicit assumption throughout that $H_{t}$ is a function of the stochastic gradient function $g_{t}$ defined in \eqref{eqn:equilirbium_gradient}. Such functions are typically of the form

\begin{equation}
\label{eqn:gradient_estimator}
 H_{t}(z) = 
    \begin{cases} 
    g_{t}(z,w_{1}), \quad w_{1}\sim D_t(z), \\
    \frac{1}{N}\sum_{i=1}^{N} g_{t}(z,w_{i}), \quad w_{1},\hdots,w_{N}\overset{i.i.d.}{\sim} D_t(z). \\
    \end{cases}
\end{equation}    

\noindent Then, given a starting point $z_{0}$, the stochastic primal-dual algorithm performs the update 
\begin{equation}
    z_{t+1} = \hat{\mathcal{G}}_{t}(z_{t}), \ \text{where} \ \hat{\mathcal{G}}_{t}(z_{t}) = \Pi_{\zDomain}\left( z_{t}-\eta H_{t}(z_{t})\right)
\end{equation}

Crucial to our analysis will be providing reasonable assumptions regarding the quality of the gradient estimator $H_{t}$. The case where $H_{t}(z) = g_{t}(z,w_{1})$ is particularly appealing in applications such as competitive markets,  strategic classification, etc.,  where $g_{t}(z,w_{1})$ can be computed using an observation of $w$ (in our example in competitive markets, we would observe the demands $a$ and $b$). 

We are interested in providing results in expectation as well as high-probability. 
A common assumption throughout the literature is to use a sub-Gaussian error model on this gradient error quantity\textemdash an observation supported by the central limit theorem when using a sufficiently large batch size $N$ in \eqref{eqn:gradient_estimator}. While this may hold in some cases, it has been observed that this requires a prohibitively large set of data while also assuming the data is of sufficiently good quality \cite{vladimirova2020sub}; this has also been observed in works on stochastic gradient methods such as in~\cite{gurbuzbalaban2020heavy,simsekli2019tail}.  A more mild assumption then is to assume a larger class of heavy-tailed distributions known as sub-Weibull distributions, which we formalize in the following.

\begin{definition}{(\textit{\textbf{Sub-Weibull Random Variable}}~\cite{vladimirova2020sub})}
The distribution of a random variable $\xi$ is sub-Weibull, denoted $\xi\sim\sW(\theta,\nu)$, if there exists $\theta > 0,\nu>0$ such that 
$\norm{z}_{p}\leq \nu p^{\theta}$, 
for all $p\geq1$.
\hfill $\Box$
\end{definition}


\begin{assumption}{(\textit{\textbf{Stochastic Framework}})}
\label{assumption_six:sub_weibull}
Denote the gradient error incurred throughout the stochastic algorithm as $\xi_{t} = H_{t}(z_{t}) - G_{t}(z_{t})$. Then there exists constant $\theta,\nu>0$ and a sequence $\{\nu_{t}\}_{t\geq0}\subseteq \real_{+}$ such that the following hold:
\begin{enumerate}
    \item \textit{\textbf{Sub-Weibull Gradient Error}}. For each $t\geq 0$, $\norm{\xi_{t}}$ is a sub-Weibull random variable such that $\norm{\xi_{t}} \sim \sW(\theta,\nu_{t})$.
    \item \textit{\textbf{Bounded Variance Proxies}}. The sequence of variance proxies $\{\nu_{t}\}_{t\geq0}$ is bounded by $\nu$. \hfill $\Box$
\end{enumerate}
\end{assumption}

With this assumption, the main convergence result is stated next. 

\begin{theorem}
\label{thm:stochastic_primal_dual} 
Suppose that Assumptions \ref{assumption_one:strong}-\ref{assumption_six:sub_weibull} hold and $\frac{\varepsilon L}{\gamma}<1$. If $\eta$ satisfies the bound in \ref{thm:convergence_stepsize_condition} then the following hold:
\begin{enumerate}
    \item \textit{\textbf{Expectation}}. The sequence $\{z_{t}\}_{t\geq0} $ satisfies the bound in expectation
    \begin{equation}
    \label{thm:first_moment_expectation}
        \EX[\norm{z_{t}-\bar{z}_{t}}] \leq \alpha^{t}\norm{z_{0}-\bar{z}_{0}} + (1-\alpha)^{-1}\Delta + (1-\alpha)^{-1}\eta\nu.
    \end{equation}
    for all $t\geq0$, for any initial point $z_{0}\in\zDomain$, and  $\alpha :=\sqrt{1-\eta(\gamma-\varepsilon L)}$.
    \item \textit{\textbf{High Probability}}. For any $\delta\in(0,1)$, and $t\geq0$,
    \begin{equation}
    \label{thm:first_moment_high_probability}
    \mathbb{P}\left(\norm{z_{t}-\bar{z}_{t}} \leq \alpha^{t}\norm{z_{0}-\bar{z}_{0}} + \frac{\Delta}{1-\alpha} + c(\theta)\log^{\theta}\left(\frac{2}{\delta}\right) \frac{\eta\nu}{1-\alpha}\right) \geq 1-\delta \, .
    \end{equation}
 with $c(\theta) := \left(\frac{2e}{\theta}\right)^\theta$, for any initial point $z_{0}\in\zDomain$.
\hfill $\Box$
\end{enumerate}
\end{theorem}

We note that the noise terms above are diametrically opposed functions of the step-size. While the drift term grows larger for small step-size, the gradient noise decreases for smaller step-size values. This relationship makes removing the contribution of any one source of perturbation impossible. We also note that the high-probability bound scales as $\log(\delta^{-1})$, as opposed to classical bounds derived using Markov's bound that scale as $\delta^{-1}$. 

Before proving the theorem, we provide supporting lemmas that will be used in the proof.

\begin{lemma}{(\textit{\textbf{Equivalent Characterizations}})} If $\xi$ is a sub-Weibull random variable with tail parameter $\theta > 0$, then the following characterizations are equivalent (we recall that $\norm{z}_{k} = \EX[\vert z\vert^{k}]^{1/k}$):
\begin{enumerate}
    \item[(c1)] \label{lemma:tail_probability} Tail Probability: $\exists \ \nu_{1}>0$ such that $\mathbb{P}(|z| \geq \epsilon ) \leq 2\exp(-\left(\epsilon/\nu_{1})\right)^{1/\theta}$ 
    for all $\epsilon>0$.
    \item[(c2)] \label{lemma:moment} Moment: $\exists \ \nu_{2}>0$ such that $\norm{z}_{k}\leq \nu_{2} k^{\theta} $ for all $k\geq1$.
\end{enumerate}
Moreover, if (c2) holds for a given $\nu_{2} > 0$, then (c1) holds with $\nu_{1} = \left(\frac{2e}{\theta}\right)^{\theta}\nu_{2}$. 
\hfill $\Box$
\end{lemma}

\begin{lemma}{(\textbf{\textit{Sub-Weibull Inclusion}})} 
\label{prop:inclusionSubWeibull}
If $\xi\sim\sW(\theta,\nu)$ based on (c2) and $\theta',\nu'>0$ such that $\theta\leq\theta'$ and $\nu\leq\nu'$ then $\xi\sim\sW(\theta',\nu')$.
\hfill $\Box$
\end{lemma}

\begin{lemma}{(\textbf{\textit{Sub-Weibull Closure}})} 
\label{prop:closureSubWeibull}
If $\xi_{1} \sim \sW(\theta_1,\nu_1)$, $\xi_{2} \sim \sW(\theta_2,\nu_2)$ are (possibly coupled) sub-Weibull random variables based on (c2) and $c\in\mathbb{R}$, then the following hold:
\begin{enumerate}
    \item  $\xi_{1}+\xi_{2}\sim\sW(\max \{\theta_1,\theta_2\}, \nu_1+\nu_2)$;
    \item  $\xi_{1}\xi_{2}\sim \sW(\theta_1 + \theta_2, 
    \psi(\theta_1, \theta_2) \nu_1 \nu_2)$, $\psi(\theta_1, \theta_2) := (\theta_1 + \theta_2)^{\theta_1 + \theta_2} / (\theta_1^{\theta_1} \theta_2^{\theta_2})$;
    \item $c\xi_{1}\sim \sW(\theta_1, |c|\nu_1)$. \hfill $\Box$ 
\end{enumerate}
\end{lemma}
The proofs of these lemmas can be found in \cite{vladimirova2020sub,wong2020lasso}.

\begin{proof}\emph{of \ref{thm:stochastic_primal_dual}}.
As before, we have that $\norm{z_{t+1} - \bar{z}_{t+1} } \leq  \norm{z_{t+1} - \bar{z}_{t} } + \Delta_{t}$
where
\begin{align*}
     \norm{z_{t+1} - \bar{z}_{t} } 
      & \leq \norm{(z_{t}-\bar{z}_{t})-\eta(H_{t}(z_{t})-G_{t}(\bar{z}_{t})} \\
      & = \norm{(z_{t}-\bar{z}_{t})-\eta(G_{t}(z_{t})-G_{t}(\bar{z}_{t}) - \eta\xi_{t}}  \leq \alpha\norm{z_{t} - \bar{z}_{t}} + \eta\norm{\xi_{t}}.
\end{align*}
This yields that stochastic recursion
$
\norm{z_{t} - \bar{z}_{t} } \leq \alpha^{t} \norm{z_{0} - \bar{z}_{0}} + \Delta\sum_{i=0}^{t} \alpha^{i} + \eta\sum_{i=0}^{t} \alpha^{i} \norm{\xi_{t-i}}
$. 
Recall that when $\eta$ satisfies the condition in \eqref{thm:convergence_stepsize_condition}, $\alpha<1$. Hence assuming this fact and taking the expectation of both sides yields
$$
\EX\norm{z_{t} - \bar{z}_{t} } \leq \alpha^{t} \norm{z_{0} - \bar{z}_{0}} + \frac{\Delta}{1-\alpha} + \eta\sum_{i=0}^{t} \alpha^{i} \EX\norm{\xi_{t-i}}
$$
so that the result in \eqref{thm:first_moment_expectation} follows. To prove the result in \eqref{thm:first_moment_high_probability}, we denote $e_{t} = \norm{z_{t}-\bar{z}_{t}}$, $\omega_{t} = \alpha_{t}\norm{z_{0} - \bar{z}_{0}} + \Delta(1-\alpha)^{-1}$, and $\sigma_{t} = \eta\sum_{i=0}^{t}\alpha^{i}\xi_{t-i}$. Observe that, due to our closure properties, 
$$
\norm{\sigma_{t}}_{p} \leq \sum_{i=0}^{t} \alpha^{i} \EX[ \norm{\xi_{t}}]^{p}]^{1 / p} \leq \frac{\eta\nu}{1-\alpha}p^{\theta}
$$
for any $p \geq 1$ and hence $\sigma_{t}\sim \sW(\theta,\eta\nu(1-\alpha)^{-1})$. It follows from Lemma \ref{lemma:tail_probability} (c1) that 
\begin{equation}
        \mathbb{P}\left(\sigma_{t} \geq \epsilon \right) \leq 2\exp \left(-\frac{\theta}{2e}
        \left( \frac{(1-\alpha)\epsilon}{\eta\nu} \right)^{\frac{1}{\theta}} \right),
\end{equation}
and setting the right hand side above equal to $\delta >0$ yields $\epsilon  = c(\theta)\log^{\theta}\left(\frac{2}{\delta}\right)\eta\nu(1-\alpha)^{-1}$.  Now, observe that our stochastic recursion implies that for any $a>0$, $\mathbb{P}(\omega_{t} + \sigma_{t} \geq a) \geq \mathbb{P}(e_{t} \geq a)$.
It follows that setting $a = \omega_{t} + \epsilon$ yields $\mathbb{P}(e_{t} \leq \omega_{t} + \epsilon ) \geq \mathbb{P}(\omega_{t}+\sigma_{t} \leq \omega_{t} + \epsilon) = \mathbb{P}(\sigma_{t}\leq \epsilon)\geq 1-\delta$, 
thus the result follows.
\end{proof}

\section{Numerical Simulations on Electric Vehicle Charging}

\begin{figure}[t!]
\label{fig:1}
 \centering
 \begin{minipage}[b]{0.48\linewidth}
    \centering
    \includegraphics[width=\textwidth]{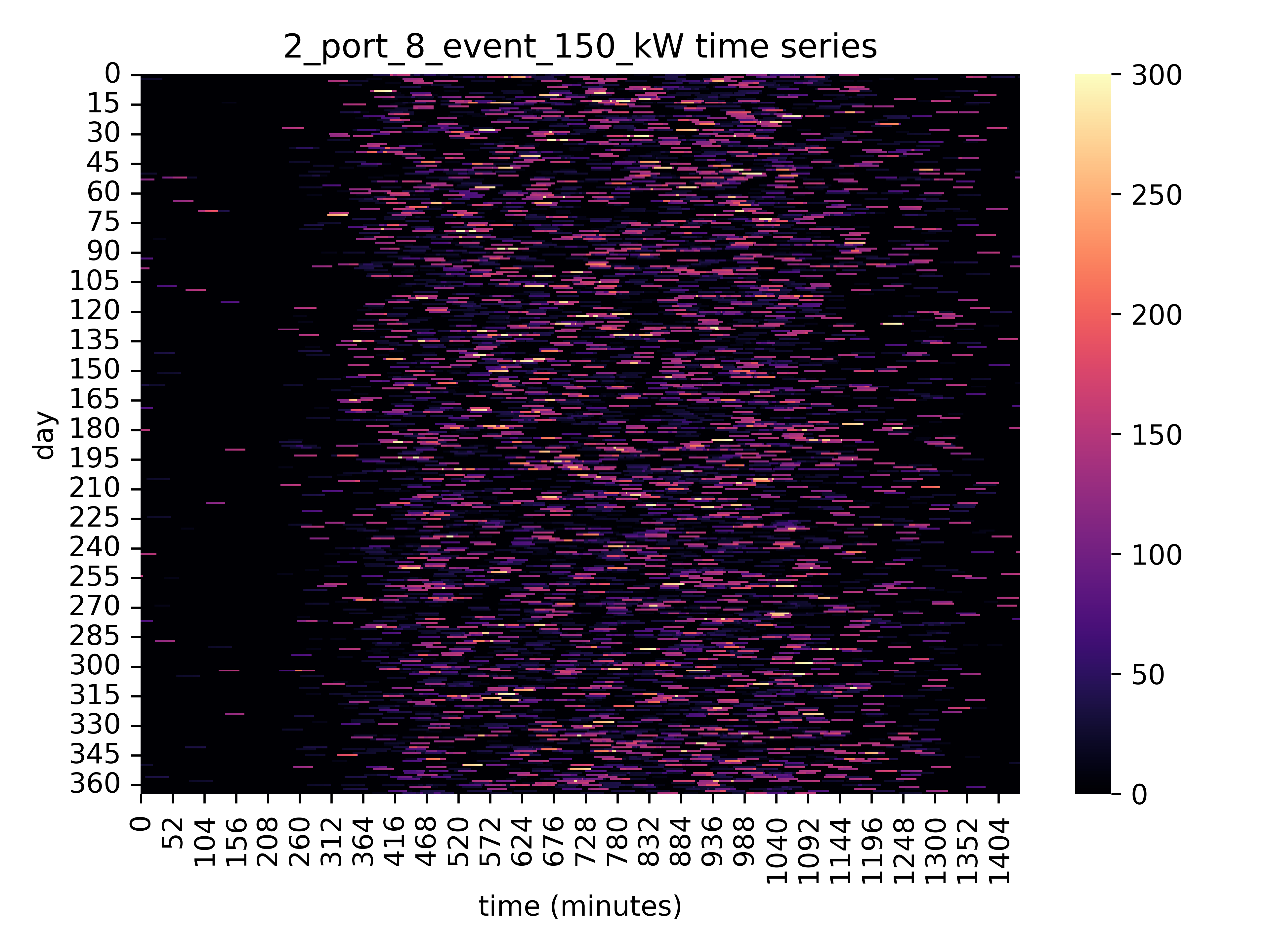}
    \small{(a)}
 \end{minipage}
 \quad
 \begin{minipage}[b]{0.48\linewidth}
    \centering
    \includegraphics[width=\textwidth]{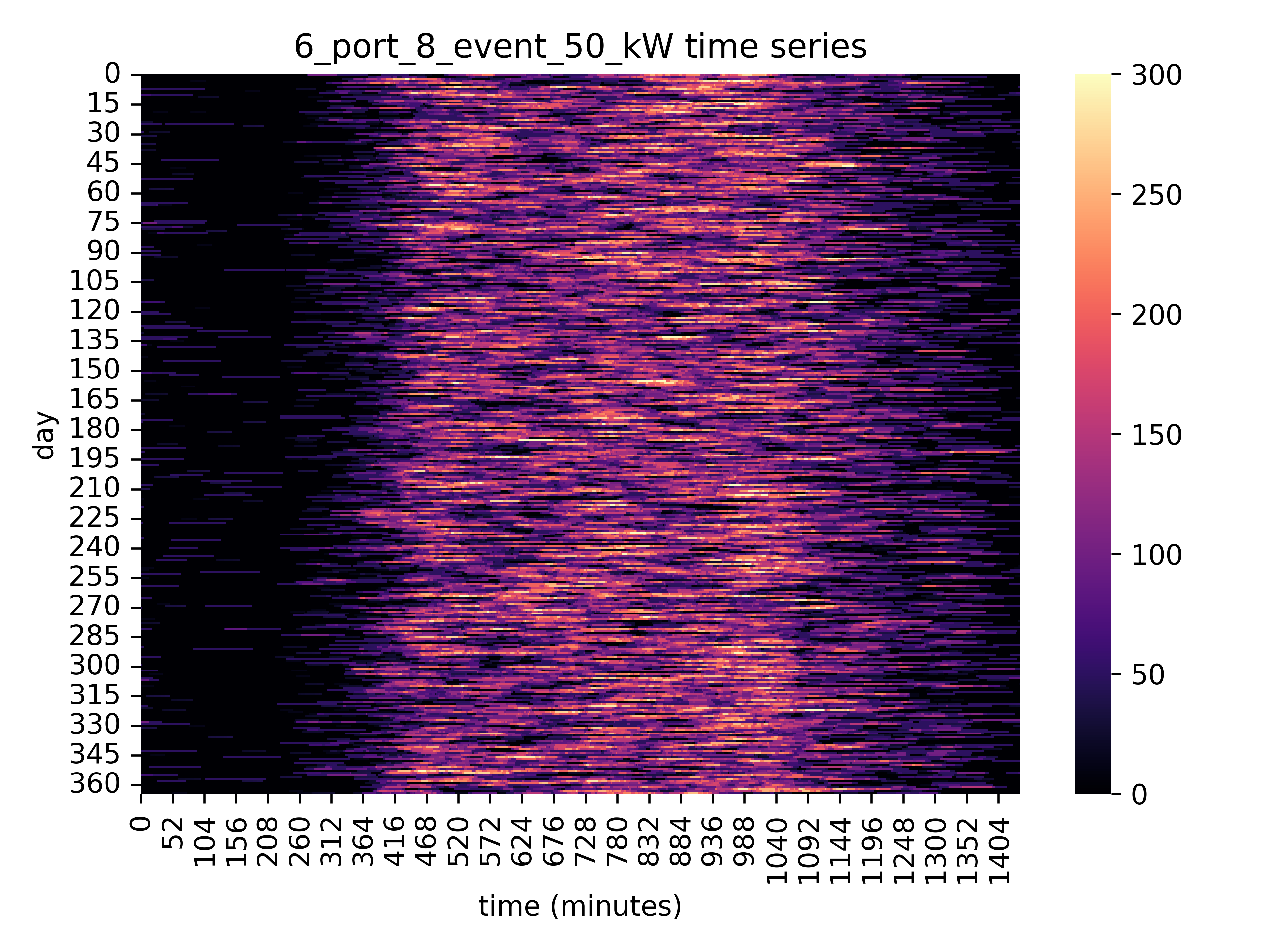} \small{(b)}
\end{minipage}
\caption{Demand time series visualization: horizontal axis is time of day, vertical axis is the day of the year between 1 and 365. Brightness indicates intensity of the demand.}
\vspace{-.5cm}
\end{figure}

\begin{figure}[t]
\label{fig:2}
 \centering
 \begin{minipage}[b]{0.48\linewidth}
    \centering
    \includegraphics[width=\textwidth]{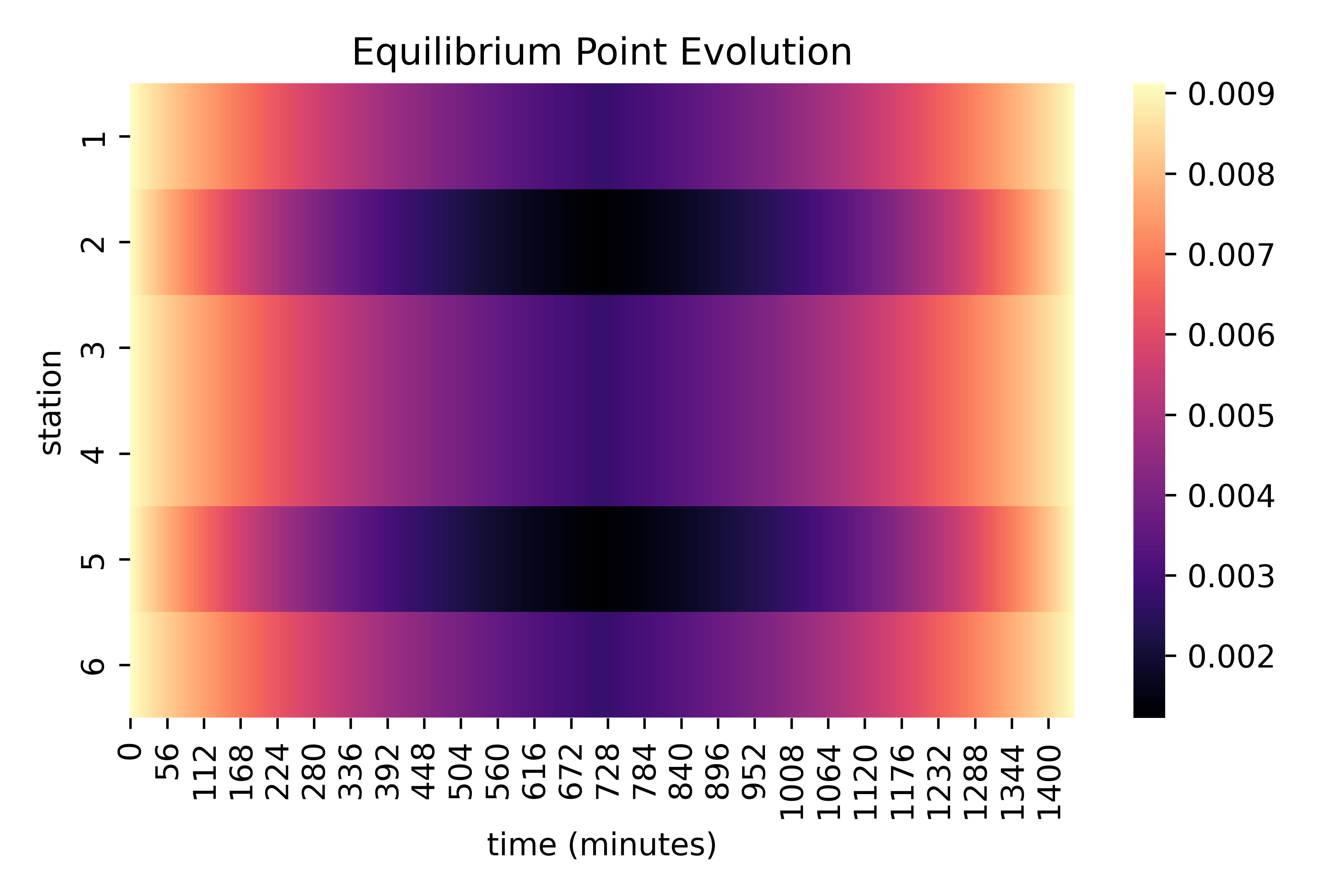}
    \small{(a)}
 \end{minipage}
 \enspace
 \begin{minipage}[b]{0.48\linewidth}
    \centering
    \includegraphics[width=\textwidth]{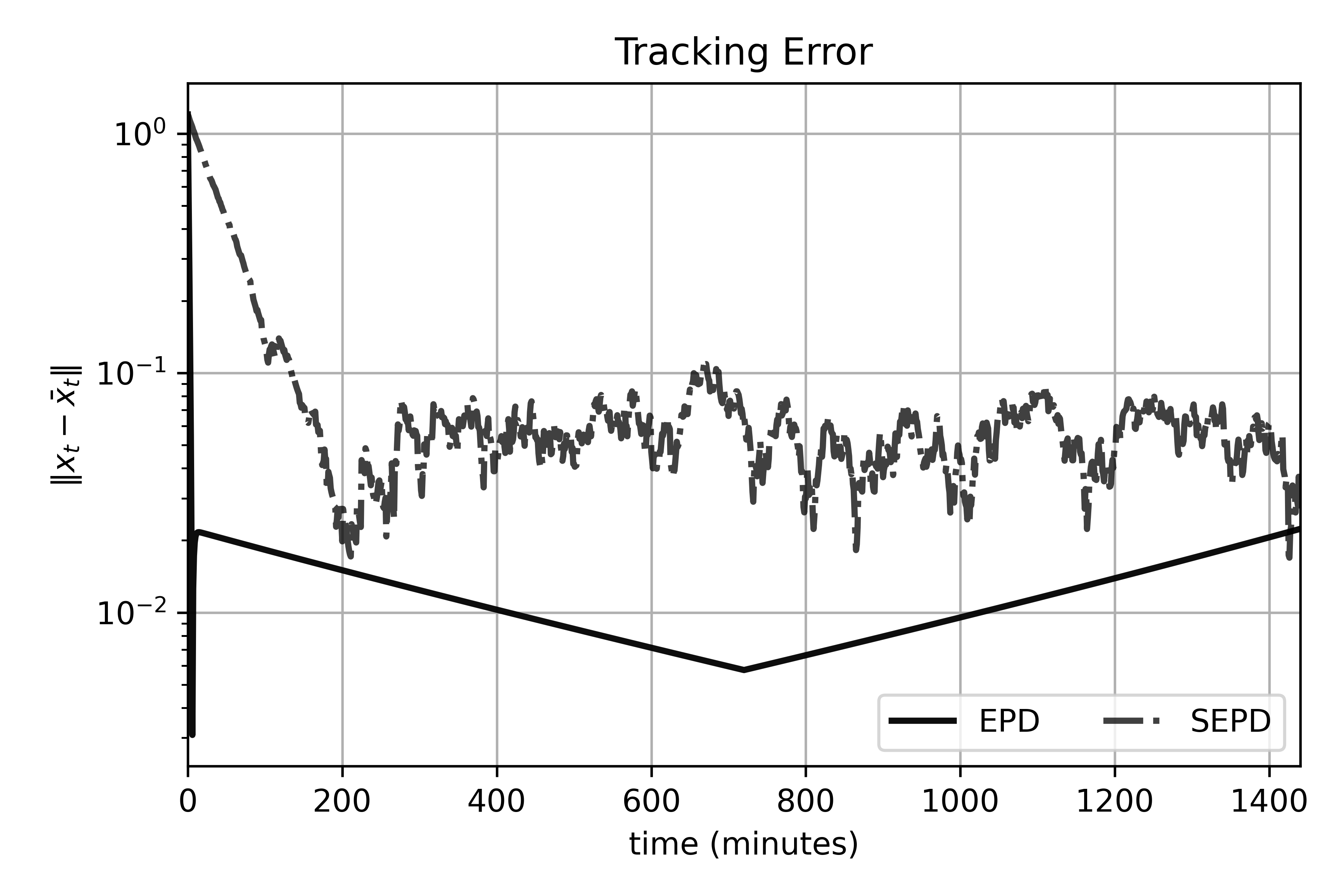} \small{(b)}
\end{minipage}
\caption{Results: in (a) we depict the evolution of the equilibrium points over the time horizon plotted in absolute value. In (b), we depict the tracking error for both algorithms. }
\vspace{-.5cm}
\end{figure}

In this section, we provide a demonstration of an online electric vehicle charging market, as described in \eqref{application:charging_market}, with time series demand data from \cite{gilleran2021electric}. The data describes a years of worth of electricity demand with entries for each minute of the year. Each file represents a different type of charging station positioned near commercial uses with varying number of ports (2 or 6), frequency of use (2, 8, or 16 event), and port power output (50, 150, or 350 kW). We randomly allocate each provider with three $8$-event stations and draw samples from each day of the year. The demand data is normalized by first subtracting the mean across each minute and dividing by the variance. 
A representative example of the raw data is provided in Figure~1, with time in minutes along the horizontal axis, day of the year along the vertical, and color intensity representing demand value. The price elasticity is dictated by the function $h_{t}(p) = (-c(p)/m \vert t-m \vert + c(p))$ where $p$ is the station's port power and $c(p)$ is given by $c(p) = 0.3$ for $p\in\{50,150\}$ and $c(p) = 0.5$ for $p=350$. The elasticity matrices are then given by $(A_{1}^{t})_{ij} = -h_{t}(p_{i})\delta_{i,j}$, $(B_{1}^{t})_{ij} = -h_{t}(p_{i})\delta_{i,j}$ for $i\in[3]$ where $p_{i}$ is the power of each port at the $i$th station belonging to the provider and $B_{2}^{t}=-B_{1}^{t}$ and $A_{2}^{t} = -A_{1}^{t}$. For the sake of simplicity, we consider service providers with unit charging speed utility rates and zero location-based utility. From this we conclude that for all $t$, $G_{t}$ is $1$-strongly monotone and $1$-Lipschitz. Hence our results apply provided that $\eta < 1 /3$.

We compute the equilibrium points by executing a batch primal-dual algorithm for 2000 iterations with a step size of $\eta = 0.001$. We then run the online primal-dual and stochastic primal-dual algorithms over each minute of the time series data and plot the distance to the solutions in Figure~2. We observe that the primal-dual algorithm is capable of reasonably tracking the trajectory. The noise incurred by the stochastic algorithm clearly prevents it from having identical performance, however the trajectory does decrease to an acceptable level after overcoming transient behavior for approximately 200 time steps.



\printbibliography
\end{document}